\documentclass[12pt]{amsart}

\usepackage{amsmath,amssymb}
\usepackage{amsfonts}
\usepackage{amsthm}
\usepackage{mathtools}

\title[Dynamics on character varieties]{Dynamics on nilpotent character varieties}

\author[J-P. Burelle]{Jean-Philippe Burelle}
\address{D\'epartement de math\'ematiques, Universit\'e de Sherbrooke, Sherbrooke, QC J1K 2R1, Canada}
\email{j-p.burelle@usherbrooke.ca}

\author[S. Lawton]{Sean Lawton}
\address{Department of Mathematical Sciences, George Mason University, 4400 University Drive, Fairfax, Virginia 22030, USA}
\email{slawton3@gmu.edu}

\newcommand{\bZ}{\mathbb{Z}}
\newcommand{\bR}{\mathbb{R}}
\newcommand{\GL}{\mathsf{GL}}
\newcommand{\SL}{\mathsf{SL}}
\newcommand{\Hom}{\mathsf{Hom}}
\newcommand{\Aut}{\mathsf{Aut}}
\newcommand{\Out}{\mathsf{Out}}

\newcommand{\X}{\mathsf{X}}

\newcommand{\SU}{\mathsf{SU}}
\newcommand{\rk}{\mathrm{rank}}
\newcommand{\norm}[1]{\left\lVert#1\right\rVert}

\newtheorem{theorem}{Theorem}[section]
\newtheorem{prop}[theorem]{Proposition}
\newtheorem{lem}[theorem]{Lemma}
\newtheorem{cor}[theorem]{Corollary}
\newtheorem{ex}[theorem]{Example}
\newtheorem{conjecture}[theorem]{Conjecture}

\theoremstyle{definition}
\newtheorem{remark}[theorem]{Remark}
\newtheorem{defn}[theorem]{Definition}
\newtheorem{question}[theorem]{Question}

\makeatletter
\@namedef{subjclassname@2020}{\textup{2020} Mathematics Subject Classification}
\makeatother



\begin{document}

\begin{abstract}
Let $\Hom^{0}(\Gamma,G)$ be the connected component of the identity of the variety of representations of a finitely generated nilpotent group $\Gamma$ into a connected compact Lie group $G$, and let $\X^0(\Gamma,G)$ be the corresponding moduli space. We show that there exists a natural $\Out(\Gamma)$-invariant measure on $\X^0(\Gamma,G)$ and that whenever $\Out(\Gamma)$ has at least one hyperbolic element, the action of $\Out(\Gamma)$ on $\X^0(\Gamma, G)$ is mixing with respect to this measure.
\end{abstract}

% Submission/arXiv Abstract:
% Let R(N,G) be the connected component of the identity of the variety of representations of a finitely generated nilpotent group N into a connected compact Lie group G, and let X(N,G) be the corresponding moduli space. We show that there exists a natural Out(N)-invariant measure on X(N,G) and that whenever Out(N) has at least one hyperbolic element, the action of Out(N) on  X(N,G) is mixing with respect to this measure.

\date{\today}

\maketitle

\tableofcontents

\section{Introduction}
Let $\Gamma$ be a finitely presentable group and let $G$ be a connected compact Lie group.  Then, the set of group homomorphisms $\Hom(\Gamma,G)$ is  a topological space, independent of a choice of presentation of $\Gamma$.  One way to see this is to give $\Gamma$ the discrete topology and consider the compact-open topology on the set of continuous functions $C^0(\Gamma, G)$. This endows $\Hom(\Gamma, G)\subset C^0(\Gamma, G)$ with the subspace topology.  The group $\Aut(\Gamma)$ acts on $\Hom(\Gamma, G)$ via $\alpha\cdot \rho=\rho\circ \alpha^{-1}$.  This action descends to an $\Out(\Gamma)$-action on the conjugation orbit space $\X(\Gamma,G):=\Hom(\Gamma, G)/G$ known as the $G$-character variety of $\Gamma$.

When $\Gamma$ is the fundamental group of a closed orientable surface of genus $g\geq 2$, there is an $\Out(\Gamma)$-invariant symplectic measure on $\X(\Gamma, G)$ by Goldman \cite{gold-sym}. In Goldman and Pickrell-Xia \cite{Gold6,PE1}, it was shown that $\Out(\Gamma)$ acts ergodically with respect to this measure.\footnote{Recent work has been announced that the Torelli subgroup is sufficient to establish ergodicity \cite{Bo}.} This theorem was generalized to most ``relative character varieties" by Pickrell-Xia and Goldman-Lawton-Xia \cite{PE2,GLX}, and for some $G$, it also holds for non-orientable surface groups by Palesi \cite{Palesi}.  When $\Gamma$ is a free group of rank 3 or greater, again there is an invariant measure where $\Out(\Gamma)$ is ergodic by Goldman and Gelander \cite{gold-free, gelander}.  Given the above examples, it is natural to ask:

\begin{question}
For what $\Gamma$ does there exist an $\Out(\Gamma)$-invariant measure with full support on $\X(\Gamma, G)$?  If so, does $\Out(\Gamma)$ act ergodically for all connected compact Lie groups $G$? 
\end{question}

In work by Mulase-Penkav \cite{MP} (see also \cite{Palesi}) the first question is explored given a presentation of $\Gamma$ satisfying certain conditions.  In general, $\X(\Gamma,G)$ may have many connected components, and $\Out(\Gamma)$ permutes them, fixing the component containing the trivial representation.  In the case of hyperbolic surface groups and free groups, the moduli space $\X(\Gamma, G)$ is connected and as mentioned above $\Out(\Gamma)$ acts ergodically with respect to measures first considered by Goldman.  As such, we call any $\Gamma$ {\it goldman} if $\Out(\Gamma)$ acts ergodically on the connected component of the identity $\X^0(\Gamma, G)$ with respect to an invariant measure (for any connected compact Lie group $G$). From this point of view, hyperbolic surface groups and free groups of rank greater than 2 are goldman.

It is the purpose of this short note to show that finitely generated abelian groups are goldman, and to prove many but not all finitely generated nilpotent groups are goldman as well (we note that finitely generated nilpotent groups are finitely presentable).  Precisely, we prove:

\begin{theorem}\label{thm:main}
Let $\Gamma$ be a finitely generated group and $G$ a compact connected Lie group.  If $\Gamma$ is nilpotent, then there exists a finite $\Out(\Gamma)$-invariant measure on $\X^0(\Gamma, G)$ with full support. If $\Aut(\Gamma)$ has a hyperbolic element $($in particular, if $\Gamma$ is abelian$)$, then the $\Out(\Gamma)$ action is strongly mixing.
\end{theorem}

Since strong mixing implies ergodicity, this immediately implies that the two classes of groups in Theorem \ref{thm:main} are goldman if the conditions of the theorem are satisfied.  We also give examples of non-abelian nilpotent groups that do in fact satisfy this condition, and some that do not.

We emphasize that Theorem \ref{thm:main} only requires the existence of a single hyperbolic element \footnote{In fact, we only require it has no eigenvalue a root of unity, see Theorem \ref{thm:main2}.} in $\Aut(\Gamma)$  whereas in the case of hyperbolic surface groups, a single hyperbolic element is not sufficient to obtain ergodicity by Brown \cite{Bro}.

When character varieties of nilpotent groups have other connected components than the one containing the trivial representation, the group $\Out(\Gamma)$ acts by some permutation on these components. In Section 5, we investigate the $\Out(\bZ^n)$ action on a specific class of examples where the \emph{exotic components} have been completely classified by Adem, Cohen, and G\'omez \cite{ACG}. The components in these examples are all isolated points, and we prove that the action of $\Out(\bZ^n)$ is transitive.

We also show that the action of $\Out(\Gamma)$ does not generally fix other components in the character variety by explicitly describing the action in some special cases.  Lastly, as a corollary of our main theorem, letting $\phi_t$ denote the flow associated to any noncompact $1$-parameter subgroup of $\SL(n,\bR)$, we define an associated flow $\widehat{\phi}_t$ on the space
\[ F_K = (\SL(n,\bR) \times \X^0(\bZ^n,K))/\SL(n,\bZ),\]
and prove that $\widehat{\phi}_t$ is strongly mixing.  This result should be compared to the main theorem in Forni-Goldman \cite{F-G}.

\subsection*{Acknowledgments}
Burelle thanks the Institute des Hautes \'Etudes Scientifiques for hosting him in 2019 when this work was started.  Lawton is partially supported by a Collaboration grant from the Simons Foundation, and also thanks the Institute des Hautes \'Etudes Scientifiques for hosting him in 2019 and again in 2021 when this work was advanced. We thank Lior Silberman for a very helpful conversation about nilpotent automorphisms which allowed us to expand Section \ref{section-examples}. We also thank David Fisher, Giovanni Forni, Ben McReynolds, and Mark Sapir for helpful references and/or discussions.  Lastly we thank the referee for helpful comments.

\section{Preliminaries}

The following definitions and lemmas are standard (see \cite{Quas} for a quick summary).

\begin{defn} A measure-preserving transformation $T$ is a map from a measure space $(X,\mu)$ to itself such that for each measurable subset $U$ of the space, $\mu(T^{-1}(U))=\mu(U)$. 
\end{defn}

\begin{defn}
A measure-preserving transformation $T$ is strong-mixing if, for any measurable sets $A,B$ one has $$\lim_{n\to\infty}\mu(A \cap T^{-n} B)= \mu(A)\mu(B).$$ 
\end{defn}

The following two lemmas follow directly from definitions.

\begin{lem}\label{prod-lem}
Let $T$ be strong-mixing on two finite measure spaces $X,Y$.  Then the diagonal action of $T$  on the product measure space $X\times Y$ is strong-mixing.
\end{lem}

\begin{proof}
Since the $\sigma$-algebra of measurable sets on $X\times Y$ is generated by measurable sets of the form $A\times B$ for measurable sets $A\subset X$ and $B\subset Y$, it suffices to check the definition for $A\times B$.  Indeed:
\begin{align*}
  &\lim_{n\to \infty}(\mu_X\times \mu_Y)((A_1\times A_2)\cap T^{-n}(B_1\times B_2))\\
&=\lim_{n\to \infty}(\mu_X\times \mu_Y)((A_1\times A_2)\cap (T^{-n}B_1\times T^{-n}B_2))\\
&=\lim_{n\to \infty}(\mu_X\times \mu_Y)((A_1\cap T^{-n}B_1)\times (A_2\cap T^{-n}B_2))\\
&=\lim_{n\to \infty}\mu_X(A_1\cap T^{-n}B_1)\mu_Y(A_2\cap T^{-n} B_2)\\
&=\mu_X(A_1)\mu_Y(B_1)\mu_X(A_2)\mu_Y(B_2)\\
&=\mu_X(A_1\times B_1)\mu_Y(A_2\times B_2).\qedhere
\end{align*}
\end{proof}

\begin{lem}\label{quot-lem}
Let $X$ be a finite measure space, $\pi:X\to Y$ is a measurable function where $Y$ has the push-forward measure. Let $T_X,T_Y$ be measure-preserving maps on $X,Y$ respectively, such that $\pi\circ T_X = T_Y \circ \pi$. If $T_X$ is strong-mixing on $X$, then $T_Y$ is strong-mixing on $Y$.
\end{lem}

\begin{proof}
Let $\mu_X$ be the measure on $X$ and $\mu_Y$ be the measure on $Y$.  By assumption, we have that $\mu_Y=\mu_X\circ \pi^{-1}$ and $\pi\circ T_X=T_Y\circ \pi$.  Then:
\begin{align*}
\lim_{n\to \infty}\mu_Y( A\cap T_Y^{-n}B)&=\lim_{n\to \infty}\mu_X(\pi^{-1}(A\cap T_Y^{-n}B))\\
&=\lim_{n\to \infty}\mu_X(\pi^{-1}A\cap \pi^{-1}T_Y^{-n}B)\\
&=\lim_{n\to \infty}\mu_X(\pi^{-1}A\cap T_X^{-n}\pi^{-1}B)\\
&=\mu_X(\pi^{-1}A) \mu_X(\pi^{-1}B)\\
&=\mu_Y(A) \mu_Y(B). \qedhere
\end{align*}
\end{proof}

\begin{defn}
The transformation $T : X \rightarrow X$ is said to be ergodic if for any invariant set $U\subset X$, either $\mu(U)=0$ or $\mu(X\setminus U)=0$.  
\end{defn}

It is easy to see that strong mixing implies ergodicity.  Indeed, without loss of generality assume $\mu$ is a probability measure. If $A$ is invariant, then $\mu(A) = \lim_{n\to\infty} \mu(A \cap T^{-n}A)=\mu(A)^2$, and so $\mu(A) = 0$ or $\mu(A) = 1 = \mu(X)$.

\section{Abelian Groups are Goldman}\label{sec:abelian}

Let $\Gamma$ be a finitely generated nilpotent group and let $K$ be a compact, connected Lie group.

The quotient map $\Gamma\rightarrow \Gamma/[\Gamma,\Gamma]$ induces an inclusion of representation spaces
\[\iota : \Hom(\Gamma/[\Gamma,\Gamma],K) \rightarrow \Hom(\Gamma,K).\]

Bergeron and Silberman showed that this inclusion restricts to a homeomorphism between the connected components containing the trivial representation.
\begin{theorem}[\cite{B-S}]\label{key-theorem}
  \[\iota : \Hom^0(\Gamma/[\Gamma,\Gamma],K)\rightarrow \Hom^0(\Gamma,K)\]
  is a homeomorphism.
\end{theorem}

Moreover, writing $\Gamma/[\Gamma,\Gamma] \cong \bZ^r \oplus A$ where $r = \rk(H^1(\Gamma,\bZ))$ and $A$ is a finite abelian group, any continuous deformation of the trivial representation must map $A$ to the identity since Lie groups do not have small subgroups. Therefore, $\Hom^0(\Gamma/[\Gamma,\Gamma],K)\cong \Hom^0(\bZ^r,K)$.

Since the Lie group $K$ is compact, the identity component of the character variety $\X^0(\bZ^r,K)$ has a simple description (see \cite[Remark 4]{Baird}): for any representation $\rho:\bZ^r \rightarrow K$, the images of the generators $$\rho(e_1),\dots,\rho(e_r)$$ are simultaneously diagonalizable, and so
\[\X^0(\bZ^r,K) \cong (U(1)^k)^r/W,\]
where $k = \rk(K)$ and $W$ is the Weyl group of $K$ acting diagonally on $r$ copies of the maximal torus $U(1)^k$.

Consider an automorphism $F\in \Aut(\Gamma)$.

\begin{defn}
  We say $F$ is \emph{hyperbolic} if the induced automorphism $\Gamma/[\Gamma,\Gamma]\rightarrow \Gamma/[\Gamma,\Gamma]$, also denoted by $F$, is a hyperbolic automorphism of the free abelian part $\bZ^r$. Hyperbolic automorphisms of free abelian groups are those which have no eigenvalue on the unit circle.
\end{defn}

The automorphism $F$ induces an automorphism on the character variety by pre-composition which we denote by $\widehat{F}\in\Aut(\X(\Gamma,K))$. By the above observations, $\X^0(\Gamma,K)\cong (U(1)^k)^r/W$. We can write an element of $(U(1)^k)^r$ as a $k\times r$ matrix
\[\Theta_\rho := \begin{pmatrix} \theta_{1,1} & \cdots & \theta_{1,r}\\
   \vdots & \ddots & \vdots\\
   \theta_{k,1} & \cdots & \theta_{k,r}\end{pmatrix},\]
where $\theta_{i,j} = \frac{\arg(\rho(e_j)_i)}{2\pi}\in \bR/\bZ$. Then, if the automorphism $F$ acts on the free part $\bZ^r$ via the matrix $M_F$, then $\widehat{F}$ acts by on the character variety by $\Theta_\rho \mapsto \Theta_\rho M_F$ preserving $W$-orbits.

By assumption, $M_F$ has no eigenvalues of modulus $1$ and so $\Theta_\rho \mapsto \Theta_\rho M_F$ is the diagonal action of a \emph{hyperbolic toral automorphism} on $r$ copies of the $k$-torus $U(1)^k$. 

The following lemma is well-known to dynamicists, and often left as an exercise (see for example \cite{KH}).  We include a proof here for the convenience of the reader (who may not be a dynamicist).

\begin{lem}\label{key-lem}
  An automorphism of the $n$-torus $T^n$ is strong mixing if and only if it does not have an eigenvalue that is a root of unity.
  \end{lem}
  
  \begin{proof}
    Let $M : T^n \rightarrow T^n$ be a toral automorphism with no eigenvalue a root of unity.
    Let $f,g \in L^2(T^n)$, it is sufficient to show that
    \[\int_{T^n} (f\circ M^k)g ~d\mu \xrightarrow{k\rightarrow\infty} \int_{T^n} f d\mu \int_{T^n} g ~d\mu,\]
    or equivalently, denoting the Hilbert inner product by $\langle\ ,\ \rangle$,
    \[\langle f \circ M^k, g \rangle \xrightarrow{k\rightarrow\infty} \langle f, 1\rangle \langle g, 1\rangle.\]

    This is because specializing $f,g$ to indicator functions of measurable sets gives exactly the definition of strong mixing.

    Denote by $e_{\vec{n}}(\vec{x}) = e^{2\pi i \langle \vec{n}, \vec{x}\rangle}$ the standard Fourier Hilbert basis, where $\vec{n}\in \bZ^n$. Precomposition by $M^k$ permutes these basis vectors:
    \[e_{\vec{n}}(M^k\vec{x}) = e^{2\pi i \langle \vec{n}, M^k\vec{x}\rangle}
    =e^{2\pi i \langle (M^{\dagger})^k\vec{n}, \vec{x}\rangle} = e_{(M^{\dagger})^k\vec{n}}(\vec{x}),\] where $\dagger$ is transpose.

    Moreover, the $M$ orbit of every non-zero integer vector $\vec{n}$ is unbounded, since if it were 
    bounded then it would be finite and hence $M^k$ would have a fixed vector for some $k$, but by assumption $M$ has no eigenvalue that is a root of unity.

    Therefore, for non-constant basis vectors we have
    \[\langle e_{\vec{n}}\circ T^k, e_{\vec{m}}\rangle \xrightarrow{k\rightarrow\infty} 0.\]

    This implies that for finite linear combinations of basis vectors, the statement holds.
    Now, since the Fourier series for $f$ and $g$ converge in $L^2$, if we denote the partial
    sums by $f_n$ and $g_n$ we have:
    \begin{align*}
      |\langle f \circ M^k, g\rangle - \langle f_n \circ M^k, g_n\rangle| &= |\langle (f - f_n) \circ M^k, g\rangle + \langle f_n \circ M^k, g-g_n\rangle|\\
      &\le \norm{f-f_n} \norm{g} + \norm{f_n} \norm{g - g_n}\xrightarrow{n\rightarrow\infty} 0\\
    \end{align*}
    and therefore, $\langle f \circ M^k, g\rangle \xrightarrow{k\rightarrow\infty} \langle f,1\rangle \langle g,1\rangle$ as desired.

    For the converse, if an automorphism $M : T^n \rightarrow T^n$ has an eigenvalue which is a root of unity, then $M^\dagger$ also has an eigenvalue which is a root of unity ($M$ and $M^\dagger$ have the same spectrum).  Thus, there exists $k\geq 1$ and non-zero $\vec{v}$ so that $(M^\dagger)^k\vec{v}=\vec{v}$.  Hence $M^\dagger\vec{u}=\vec{u}$ for $\vec{u}:=\sum_{i=0}^{k-1}(M^\dagger)^i\vec{v}$.  Then from our above calculations: $e_{\vec{u}}(M\vec{x})  = e_{(M^{\dagger})\vec{u}}(\vec{x})=e_{\vec{u}}(\vec{x}).$  Thus, there is a non-constant invariant measurable function.  Hence, the action cannot be ergodic (since if it was ergodic all invariant measurable functions would be constant), let alone strong mixing.
  \end{proof}

We can now prove Theorem \ref{thm:main} from the Introduction, which follows from the following slightly more general theorem:

\begin{theorem}\label{thm:main2}
Let $\Gamma$ be a finitely generated nilpotent group, and $G$ a compact connected Lie group.  Then $\Out(\Gamma)$ is strongly mixing on $\X^0(\Gamma, G)$ whenever there exists an element in $\Aut(\Gamma)$ that does not have any eigenvalue that is a root of unity.
\end{theorem}

\begin{proof}
First, if the free part of $\Gamma/[\Gamma,\Gamma]$ is empty, then $\X^0(\Gamma, G)$ is a point (the identity representation), and the theorem holds trivially.

Now assume that the free part of $\Gamma/[\Gamma,\Gamma]$ has rank $r \geq 1$, and let $F$ be an automorphism of $\Gamma/[\Gamma,\Gamma]$ without an eigenvalue that is a root of unity.  By Theorem \ref{key-theorem} and \cite[Remark 4]{Baird}), $\X^0(\Gamma,K)\cong (U(1)^k)^r/W$.

Since a toral automorphism with no eigenvalue a root of unity is strong mixing with respect to the Lebesgue measure (Lemma \ref{key-lem}), its diagonal action on a product of tori is also strong mixing with respect to the product measure (Lemma \ref{prod-lem}).

The Weyl group acts on $(U(1)^k)^r$ by Euclidean reflections, which preserve Lebesgue measure, and so the $\Out(\Gamma)$-action on $$\X^0(\Gamma,K)\cong (U(1)^k)^r/W$$ has an invariant measure, induced by the product of Lebesgue measures on $(U(1)^k)^r$. Therefore, the action of $\widehat{F}$ on $\X^0(\Gamma,K)\cong (U(1)^k)^r/W$ is strong mixing with respect to this measure by Lemma \ref{quot-lem}.
\end{proof}

\begin{remark}
 Whenever $\Gamma$ is abelian with non-trivial free part, $\Out(\Gamma)$ contains a hyperbolic element.  
\end{remark}

\begin{remark}
The ``only if'' part of Lemma \ref{key-lem} implies that if the group of automorphisms of $\Gamma$ is unipotent then a {\it single} automorphism cannot be mixing nor ergodic on $\X^0(\Gamma, G)$.
\end{remark}

\begin{remark}
 It is known that if $\Gamma$ is finitely generated and torsion-free, then $\Out(\Gamma)$ is non-trivial \cite{Ree}.   It appears to be an open problem whether there exists a \emph{finitely generated} nilpotent group with trivial outer automorphism group (which does occur if $\Gamma$ is allowed to be infinitely generated \cite{Za}).
\end{remark}

 In the next section we provide large classes of examples of non-abelian nilpotent groups that admit hyperbolic outer automorphisms, and some that do not.

\section{Examples of Nilpotent Groups that are Goldman}\label{section-examples}

If a subset $S$ of a group $G$ projects onto a generating set of the abelianization $G/[G,G]$, then $S$ is said to {\it weakly generate} $G$.  Every generating set of $G$ is a weak generating set, but the converse is not always true.  For nilpotent groups, however, the converse holds :

\begin{lem}{\cite[Lemma 5.9]{MKS}} \label{wg-lemma}Let $S$ be a subset of a nilpotent group $G$. Then
 $S$ weakly generates $G$ if and only if $S$ generates $G$.
\end{lem}

We will see this fact is very useful to understand which nilpotent groups admit hyperbolic automorphisms.

\begin{ex}\label{ex-free}
For a group $\Gamma$, the {\it lower central series} is defined by $\ell_0(\Gamma)=\Gamma$ and $\ell_{i+1}(\Gamma)=[\Gamma,\ell_i(\Gamma)]$ where $[X,Y]$ is the commutator subgroup generated by $X,Y\leq \Gamma$.  A group is nilpotent, by definition, if the series $$\Gamma=\ell_0(\Gamma)\geq \ell_1(\Gamma) \geq \ell_2(\Gamma)\geq\cdots$$ terminates, that is, for some $s$, $\ell_s(\Gamma)$ is the trivial group $1$.  When $s$ is the minimal such integer we say $\Gamma$ is nilpotent of step size $s$.  The {\it free nilpotent} group of step size $s$ and of rank $r$ is the quotient group $N_{s,r}:=F_r/\ell_{s}(F_r)$ where $F_r$ is a free group of rank $r$.

The abelianization homomorphism $N_{s,r}\to N_{s,r}/[N_{s,r},N_{s,r}]\cong \bZ^r$ induces a homomorphism $\varphi:\Out(N_{s,r})\to \GL(r,\bZ)$.  Lemma \ref{wg-lemma} implies $\varphi$ is surjective, and hence there are hyperbolic automorphisms of $N_{s,r}$ for all $s\geq 1$ and $r\geq 1$.
\end{ex}

\begin{ex}\label{ex-heisenberg}The $(2n+1)$-dimensional discrete Heisenberg group is the group $H_{2n+1}(\bZ):=\bZ^{n}\times\bZ^n\times \bZ$ with the group law:  $$(\mathbf{a}_1,\mathbf{b}_1,c_1)*(\mathbf{a}_2,\mathbf{b}_2,c_2)=(\mathbf{a}_1+\mathbf{a}_2,\mathbf{b}_1+\mathbf{b}_2,c_1+c_2+(\mathbf{a}_1\cdot \mathbf{b}_2-\mathbf{a}_2\cdot \mathbf{b}_1)).$$

Since $\{(\mathbf{0},\mathbf{0},c)\in H_{2n+1}(\bZ)\ |\ c\in \bZ\}$ is central in $H_{2n+1}(\bZ)$, and $\mathbf{a}_1\cdot \mathbf{b}_2-\mathbf{a}_2\cdot \mathbf{b}_1$ defines a symplectic form any automorphism must preserve, the abelianization homomorphism $$H_{2n+1}(\bZ)\to H_{2n+1}(\bZ)/[H_{2n+1}(\bZ),H_{2n+1}(\bZ)]$$ induces a homomorphism $\Out(H_{2n+1}(\bZ))\to \mathsf{Sp}(2n,\bZ).$ Every $M\in\mathsf{Sp}(2n,\bZ)$ defines an automorphism of $H_{2n+1}(\bZ)$ by $$M(\mathbf{a},\mathbf{b},c) := (M(\mathbf{a},\mathbf{b}) , c)$$ since it preserves the group law.  Thus $\Out(H_{2n+1}(\bZ))\to \mathsf{Sp}(2n,\bZ)$ is an epimorphism and  we conclude there exists hyperbolic elements in $\Out(H_{2n+1}(\bZ))$.
\end{ex}

\begin{remark}
In the case that $n=1$ we can say more. In \cite{heisauto}, it is shown that $\Out(H_3(\bZ)) \cong \GL(2,\bZ)\cong \SL(2,\bZ)\ltimes (\bZ/2\bZ)$.   This is consistent with our prior observations since $\SL(2,\bZ)\cong \mathsf{Sp}(2,\bZ)$ and $H_3(\bZ)\cong N_{2,2}$.
\end{remark}

\begin{cor}
The action of $\Out(\Gamma)$ on $\X^0(\Gamma,K)$ is ergodic with respect to the Lebesgue measure whenever $\Gamma$ is a free nilpotent group, or a discrete Heisenberg group.  
\end{cor}

\begin{proof}
Theorem \ref{thm:main2} and Examples \ref{ex-free} and \ref{ex-heisenberg} establish the corollary. 
\end{proof}

Moreover, there exists finitely generated nilpotent groups that do {\it not} admit hyperbolic automorphisms as the next example shows.

\begin{ex}
  Dyer \cite{Dyer} shows there exists a nilpotent Lie algebra whose automorphism group is strictly unipotent.  As Dyer says in his introduction ``for connected simply connected nilpotent Lie groups the choice of group or algebra language is a matter of taste.''  Consequently, by taking $\bZ$-points of the corresponding Lie group that Dyer considers, we have a finitely generated nilpotent $\Gamma$ whose automorphism group is strictly unipotent, hence there can be no hyperbolic automorphism in that case.
\end{ex}

\begin{remark}
Dyer \cite{Dyer} proves that all automorphisms in his example are unipotent by showing they are {\it trivial on the abelianization}.  Thus, Dyer's example shows there are examples where the corresponding action on $\X^0(\Gamma, G)$ is not ergodic since the induced automorphisms are all trivial. Hence, we know that the class of finitely generated nilpotent groups are not all goldman.
\end{remark}

The next example generalizes both abelian and Heisenberg groups.

\begin{ex}\label{ex:SilbermanLU}
The following example was communicated to us by Lior Silberman.  Let $G$ be a connected reductive affine algebraic group defined over $\bZ$, and $P<G$ a parabolic subgroup. Then $P=LU$ where $L$ is the Levi factor of $P$, and $U$ is the unipotent radical of $P$.  Since $P\cong L\ltimes U$, $L$ acts on $U$ by conjugation $($and hence are automorphisms$)$.  In general, unipotent Lie groups $($having all elements unipotent$)$ are nilpotent groups. 

The set of $\bZ$-points of $U$ forms a finitely generated nilpotent group, and the set of $\bZ$-points of $L$ acts on $U$ by $($non-inner$)$ automorphisms.  Since $L$ is reductive, the locus of $\bZ$-points generally contains hyperbolic elements.
\end{ex}

\begin{ex}
As pointed out to us by David Fisher, Nilmanifolds $N$ admitting an Anosov diffeomorphism $\alpha$ give examples of nilpotent groups $($the fundamental group of $N)$ admitting a hyperbolic automorphism $($induced by $\alpha)$.  See for example \cite{LW} and the references therein.
\end{ex}

The fact that there exist examples of finitely generated nilpotent $\Gamma$ where $\Out(\Gamma)$ does not satisfy the conditions of Theorem \ref{thm:main2} does not a priori rule out the possibility that the action on $\X^0(\Gamma, G)$ is ergodic.  So we pose:

\begin{question}\label{question2}Is there an example of a finitely generated nilpotent group $\Gamma$ such that $\Out(\Gamma)$ is unipotent and the action on $\X^0(\Gamma, G)$ is ergodic?
\end{question}

Given the example by Dyer, we suspect the answer is no.  That, with Theorem \ref{thm:main2}, motivates us to formulate:

\begin{conjecture}
 Let $\Gamma$ be a finitely generated nilpotent group, and $G$ a connected compact Lie group.  Then $\Out(\Gamma)$ acts ergodically on $\X^0(\Gamma,G)$ with respect to the Lebesgue measure if and only if $\Out(\Gamma)$ contains an element with no eigenvalue a root of unity.
\end{conjecture}

\section{Exotic Components in the Abelian Case}

In this section we describe the action on the non-identity components in special cases where they can be explicitly computed, as investigated in \cite{ACG}.  Let $p$ be prime and $\bZ_p$ be the cyclic group of order $p$.  We will use the same notation for center of $\SU(p)$ which is realized as the subgroup of scalar matrices with values $p$-th roots of unity.  Let $\Delta(p)$ be the diagonal embedding of $\bZ_p$ in $(\bZ_p)^m=Z(\SU(p)^m)$.  Let $G_{m,p}:=\SU(p)^m/\Delta(p)$.  For example, $G_{1,p}=\mathrm{PU}(p)$.

The main theorems in \cite{ACG} explicitly describe the components in $\Hom(\bZ^r,G_{m,p})$ and $\X(\bZ^r,G_{m,p})$.  In particular, $\X(\bZ^r,G_{m,p})$ consists of the identity component $$\X^0(\bZ^r,G_{m,p})\cong T^r/W\cong (S^1)^{(p-1)rm}/(\Sigma_p)^m$$ where $\Sigma_p$ is the symmetric group on $p$ letters, and $$\frac{p^{(m-1)(r-2)}(p^r-1)(p^{r-1}-1)}{p^2-1}$$ discrete points.

A priori, the only component we can generally say is fixed by $$\Out(\bZ^r)\cong \GL(r,\bZ)$$ is the identity component.  Each isolated point in $\X(\bZ,G_{m,p})$ corresponds to a path-component in $\Hom(\bZ^r,G_{m,p})$; all of which are isomorphic to the homogeneous space $\SU(p)^m/(\bZ_p^{m-1}\times E_p)$ where $E_p\subset \SU(p)$ is isomorphic to the quaternion group $Q_8$ if $p$ is even and the group of triangular $3\times 3$ matrices over the $\bZ_p$, with 1's on the diagonal when $p$ is odd (either way it is ``extra-special'' of order $p^3$).

The action of $\GL(r,\bZ)$ on the character variety lifts to an action on $\Hom(\bZ^r, G_{m,p})$ and so we study this to understand the dynamics.

Pairs of elements $x,y \in \SU(p)^m$ such that the commutator $[x,y] = xyx^{-1}y^{-1}$ belongs to the central group $\Delta(p)\cong \bZ_p$ are called \emph{almost-commuting} elements. An $r$-tuple of pairwise almost-commuting elements $x_1,\dots,x_r\in \SU(p)^m$ commutes in the quotient $G_{m,p}$ and therefore defines a representation of $\bZ^r$.

The two results from \cite{ACG} that we will use are the following:

\begin{lem}[Lemma 8 of \cite{ACG}]\label{ACGLemma8}
  Let $x,y\in \SU(p)^m$ be a pair of almost-commuting elements with $[x,y] = c \neq 1$. Assume $z$ almost-commutes with $x$ and $y$, and let $[x,z]=c^b$ and $[y,z]=c^a$ for integers $0\le a,b <p$. Then, there exists $w\in Z(SU(p)^m)$ such that $z = w x^{-a}y^b$.  
\end{lem}

\begin{cor}
  Let $x_1,\dots,x_r \in \SU(p)^m$ be almost-commuting elements, $Z_{ij}:=[x_i,x_j]$ for all $1\leq i,j\leq r$, and denote $Z = (Z_{ij})$ the $r\times r$ matrix of their commutators. Then, the first two rows of $Z$ can take arbitrary values in $\Delta(p)$, and given the two first rows the rest of the matrix $Z$ is uniquely determined.
\end{cor}

In order to compute the action of an outer automorphism of $\bZ^r$ on representations, we first prove a general lemma about its action on commutators of almost-commuting elements.

Let $G$ be a group, and let $A_1,\dots,A_r \in G$ such that the commutators $[A_i, A_j] = z_{i,j} \in Z(G)$ are central.  Let $A:=(A_1,...,A_r)$.

If we use additive notation for $Z(G)$ since it is abelian, then the matrix $Z:= [A,A]$, whose entry in position $(i,j)$ is $[A_i,A_j]$, is skew-symmetric.

For $M\in \GL(r,\bZ)$, define $A_i^M = A_1^{M_{i,1}}A_2^{M_{i,2}} \cdots A_r^{M_{i,r}}$, and $A^M = (A_1^M,\dots,A_r^M)$.
\begin{lem}
  \[[A^M,A^M] = M^t Z M\]
  \begin{proof}
    By a direct computation,
    \begin{align*}
      A_i^M A_j^M = A_1^{M_{i,1}}A_2^{M_{i,2}} \cdots A_r^{M_{i,r}} A_1^{M_{j,1}}A_2^{M_{j,2}} \cdots A_r^{M_{j,r}} =\\ \left(\sum_{k=1}^{r-1} M_{i,r}M_{j,k} z_{r,k}\right) A_1^{M_{i,1}}A_2^{M_{i,2}} \cdots A_{r-1}^{M_{i,r-1}} A_1^{M_{j,1}}A_2^{M_{j,2}} \cdots A_r^{M_{j,r}}A_r^{M_{i,r}}\\
      =\cdots\\
      =\left(\sum_{l=1}^{r}\sum_{k=1}^{r} M_{i,l}M_{j,k} z_{l,k}\right) A_j^M A_i^M\\
      =(M^t Z M)_{i,j} A_j^M A_i^M
    \end{align*}
  \end{proof}
\end{lem}

Now let us assume that the center $Z(G)$ is a finite abelian group. Notice that the action of $\GL(r,\bZ)$ on the skew-symmetric matrix $Z$ above is the usual action on skew-symmetric bilinear forms. Therefore, for any such matrix $Z$ with values in $Z(G)$, there exists
$M\in \GL(r,\bZ)$ such that $M^t Z M$ has the standard form:
\[M^t Z M = \begin{pmatrix}
  0 & a_1 & 0 & 0 & \cdots & 0\\
  -a_1 & 0 & 0 & 0 & \cdots & 0\\
  0 & 0 & 0 & a_2 & \cdots & 0 \\
  0 & 0 & -a_2 & 0 & \cdots & 0\\
  \vdots & \vdots & \vdots &\vdots & \ddots & \vdots\\
  0 & 0  & 0 & 0 & \cdots & 0\\
\end{pmatrix},\] where $a_i\in Z(G)$ for $1\leq i\leq \lfloor r/2\rfloor$.

For the specific case of $G=G_{m,p}$, the center is $\bZ_p$ for some prime $p$.

\begin{theorem}
  Suppose $r\ge 3$. Then, $\Out(\bZ^r) = \GL(r,\bZ)$ acts transitively on the exotic components
  of $\X(\bZ^r,G_{m,p})$.
\end{theorem}

  \begin{proof}
    The main results of \cite{ACG} show that the exotic components are uniquely determined by the first
    two rows of the matrix of commutators $Z$, and that these two rows can take any possible values
    in $\bZ_p$.

    Let $x_1,\dots,x_r\in \SU(p)^m$ be almost-commuting elements with matrix of commutators $Z$, and assume that they do not all commute. By the remarks above, up to the action of $\GL(r,\bZ)$ we may assume that the first two rows of $Z$ are
    \[\begin{pmatrix}
      0 & a_1 & 0 & 0 & \cdots & 0\\
      -a_1 & 0 & 0 & 0 & \cdots & 0\\
    \end{pmatrix},\]
    where $a_1\neq 0$.

    By Lemma \ref{ACGLemma8}, if those are the first two rows of the matrix of commutators, then every other entry of the matrix $Z$ must be $0$.

    Define
    \[M = \begin{pmatrix}
      a_1^{-1} & 0 &\cdots & 0 & 0\\
      0 & 1 & \cdots & 0 & 0\\
      \vdots & \vdots & \ddots & \vdots & \vdots\\
      0 & 0 & \cdots & 1 & 0\\
      0 & 0 & \cdots & 0 & a_1\\
    \end{pmatrix} \in \GL(r,\bZ_p),\]
    and let $\hat{M} \in \GL(r,\bZ)$ be a lift of $M$. Then, the first two rows of $\hat{M}^t Z \hat{M}$ are
    \[\begin{pmatrix}
      0 & 1 & 0 & 0 & \cdots & 0\\
      -1 & 0 & 0 & 0 & \cdots & 0\\
    \end{pmatrix},\]
    and so every exotic component is in the orbit of the one determined by this matrix of commutators, and the action of $\GL(n,\bZ)$ is transitive.

  \end{proof}

\begin{remark}
  In characteristic 2, the classification of skew-symmetric bilinear forms is different than in odd characteristic. However, in our setting, we know that the diagonal of the matrix $Z$ is $0$ since a group element always commutes with itself. This implies that the normal form for $Z$ given above applies to all characteristics.
\end{remark}

The following proposition describes the case $r=2$.
\begin{prop}
  If $p > 2$, the group $\Out(\bZ^2) = \GL(2,\bZ)$ permutes the exotic components
  $\X(\bZ^r,G_{m,p})$ in pairs. If $p=2$, it fixes every component.
  \begin{proof}
    In this case, there is only one non-trivial commutator $c \in \bZ_p$, and the action of an automorphism
    $M\in \GL(2,\bZ)$ maps $c$ to $\det(M)c$. Since the only invertible elements in $\bZ$ are $\pm 1$, the
    orbit of the exotic component with commutator $c$ is made of the two components with commutators $\pm c$, which are distinct when $p>2$ but equal when $p=2$.
  \end{proof}
\end{prop}

\section{Mixing flows associated to outer automorphisms}
In \cite{F-G}, Forni and Goldman define a bundle $\mathbb{U}\mathfrak{M}(S)$ over the unit sphere bundle of the Riemann moduli space associated to a surface $S$, with fibers the character variety $\X(\pi_1(S),G)$ into a Lie group $G$. They define a flow on $\mathbb{U}\mathfrak{M}(S)$ lifting the Teichmüller geodesic flow in order to create a continuous version of the mapping class group dynamics on character varieties of surfaces, and prove that this flow is ergodic.

By analogy with this construction, we consider the flow of a $1$-parameter subgroup on $\SL(n,\bR)/\SL(n,\bZ)$. The group $\SL(n,\bZ)$ is a lattice in $\SL(n,\bR)$. We can therefore use Moore's ergodicity theorem \cite{Moore}:

\begin{theorem}[Moore's ergodicity theorem]\label{thm:Moore}
  Let $\Gamma\subset G$ be an irreducible lattice. If $H\subset G$ is a closed noncompact subgroup, then $H$ acts ergodically on $G/\Gamma$ $($with respect to the Haar measure$).$
\end{theorem}

Let $\phi_t$ denote the flow associated to any noncompact $1$-parameter subgroup of $\SL(n,\bR)$. Then, by Theorem \ref{thm:Moore} the $\phi_t$ action is ergodic on $\SL(n,\bR)/\SL(n,\bZ)$.

Let $K$ be a compact connected Lie group and consider the associated space
\[ F_K := (\SL(n,\bR) \times \X^0(\bZ^n,K))/\SL(n,\bZ),\]
where $\SL(n,\bZ)$ acts diagonally by right multiplication on the first factor and by outer automorphisms on the second. Then, $F_K$ is a flat bundle over the quotient $\SL(n,\bR)/\SL(n,\bZ)$ with fiber $\X^0(\bZ^n,K)$.

We lift the flow $\phi_t$ to a flow $\widehat{\phi}_t$ on $F_K$ which is trivial on the second factor, and equal to $\phi_t$ on the first. This flow is a continuous analog of the outer automorphism action on the character variety $\X^0(\bZ^n,K)$.

Following \cite{F-G} exactly we have:

\begin{theorem}\label{thm:phitstrongmixing}
  The action of $\widehat{\phi}_t$ on $F_K$ is strong mixing.
  \end{theorem}
  
  \begin{proof}
    The product of the Haar measure on $\SL(n,\bR)$ with the Lebesgue measure defined in Section \ref{sec:abelian} on $\X^0(\bZ^n,K)$ defines a $\widehat{\phi}_t$ invariant measure on $\SL(n,\bR) \times \X^0(\bZ^n,K)$, which passes down to an invariant probability measure on the quotient $F_K$.

    The $\phi_t \times \SL(n,\bZ)$-action on $\SL(n,\bR)$ is ergodic. If we act by $\phi_t \times \SL(n,\bZ)$ on $\X^0(\bZ^n,K)$ with the first factor acting trivially, we obtain a strongly mixing action by Theorem \ref{thm:main}. The \emph{multiplier criterion} \cite{glasnerweissmixing} for weak mixing states the action of a group on a product $X\times Y$ of probability spaces is ergodic if the action on $X$ is ergodic and the action on $Y$ is weakly mixing. We deduce that the $\widehat{\phi}_t \times \SL(n,\bZ)$-action on $\SL(n,\bR) \times \X^0(\bZ^n,K)$ is ergodic, and hence that the $\widehat{\phi}_t$ action on $F_K$ is ergodic.
    
    Now, the flow $\widehat{\phi}_t$ extends to an action of the whole Lie group $\SL(n,\bR)$. As a corollary of the Howe-Moore theorem on vanishing of matrix coefficients \cite{HoMo}, any measure-preserving ergodic action of $\SL(n,\bR)$ on a probability space is strongly mixing and restricts to a strongly mixing action on any noncompact closed subgroup (see for instance \cite{BekkaMayer}). We conclude that the flow $\widehat{\phi}_t$ is strongly mixing.
    
  \end{proof}

  More generally, this construction applies to the setting of Example \ref{ex:SilbermanLU}. Let $P = LU < G$ be the Levi decomposition of a parabolic subgroup of a reductive algebraic group $G$ defined over $\bZ$. Suppose moreover that the identity component of $L$ has no non-trivial $\mathbb{Q}$-characters (in particular, if $L$ is semisimple).
  By the Borel-Harish-Chandra theorem \cite{BorelHarishChandra}, the $\bZ$-points  $L_\bZ \subset L$ form an arithmetic lattice, and as stated in the example $L_\bZ$ acts by automorphisms on $U_\bZ$ via conjugation. Given a choice of noncompact $1$-parameter subgroup $H\subset L_\bR$, we can therefore associate a flow $\widehat{\phi}^P_t$ on the space
  \[F^P_K = (L_\bR \times \X^0(U_\bZ,K))/L_\bZ.\]
  Assuming that $L_\bZ$ is irreducible and contains a hyperbolic automorphism of $U_\bZ$, the same proof as Theorem \ref{thm:phitstrongmixing} shows:
  \begin{theorem}
    The action of $\widehat{\phi}^P_t$ on $F^P_K$ is strongly mixing.
  \end{theorem}

  With the above definitions, we have constructed a flow which reflects the dynamical properties of action of $\SL(n,\bZ)$ by outer automorphisms on $\bZ^n$, and more generally of the action of $L_\bZ$ on $U_\bZ$. A natural invariant of a flow is its entropy, and so we ask:
  \begin{question}
    Given a choice of 1-parameter flow $\phi_t$, are there natural bounds for the entropy of the flow $\widehat{\phi}_t$?
  \end{question}

\end{document}